\documentclass[a4paper]{amsart}
\usepackage[cp1251]{inputenc}
\usepackage[english]{babel}
\usepackage{verbatim}

\usepackage{amsmath}
\usepackage{amsfonts}

\usepackage{amscd}
\usepackage{amssymb}
\usepackage{latexsym}

\newcommand{\Var}{\mathop{\rm Var}\nolimits}
\newcommand{\sign}{\mathop{\rm sign}\nolimits}

\newcommand{\seq}[2]{#1_1,\dotsc,#1_#2}
\newcommand{\ind}{\mathbf{1}}
\newcommand{\D}{\mathbf{D}\,}
\renewcommand{\Re}{\mathop{\rm\bf Re}\nolimits}

\renewcommand{\P}{\mathbf{P}\,}
\newcommand{\E}{\mathbf{E}\,}

\newcommand{\Rd}{\mathbb{R}^{d}}
\newcommand{\R}{\mathbb{R}}

\newcommand{\Sd}{{\mathbb S}^{d-1}}

\newtheorem{theorem}{Theorem}
\newtheorem{lemma}{Lemma}
\newtheorem*{corollary}{Corollary}
\newtheorem*{remark}{Remark}
\newtheorem{example}{Example}

\begin{document}

%\special{papersize=348mm,210mm}   
\title{On random surface area}
\author{Ildar Ibragimov, Dmitry Zaporozhets}
\begin{abstract}
Consider a random smooth Gaussian field $G(x):F\to\mathbb{R}$, where $F$ is a compact in $\mathbb{R}^d$. We derive a formula for average area of a surface generated by the equation $G(x)=0$ and give some applications. As an auxiliary result we obtain an integral expression for area of a surface induced by zeros of a \emph{non-random} smooth field.

\emph{Keywords:} random Gaussian field, surface area, Favard measure, coarea formula, Rice formula.
\end{abstract}

\thanks{Partially supported by RFBR (08-01-00692, 10-01-00242), RFBR-DFG (09-0191331), NSh-4472.2010.1, CRC 701 ``Spectral Structures and Topological Methods in Mathematics''} 
\maketitle

\setlength{\hoffset}{0.5mm}

\section{Results}
Consider a compact set $F\subset\Rd$. By $\partial F$ denote the boundary of $F$. We assume that the area of $\partial F$ is finite (the notion of area is defined below). Let $G(x):F\to\R$ be a random Gaussian field. Put $m(x)=\E G(x)$ and $\sigma^2(x)=\Var G(x)$. Here and below we assume that $\sigma(x)>0$ for all $x\in F$ and $G\in\mathcal{C}^1(F)$ a.s. It is known that the supremum of a continues Gaussian field defined on a compact is summable  (see~\cite{LS70}). Therefore, by Kolmogorov's Theorem on differentiation of mathematical expectations with respect to a parameter (see \cite{aK98}), we have $m,\sigma\in\mathcal{C}^1(F)$. Let $G_i',\sigma_i'$ denote partial derivatives of $G,\sigma$ with respect to $i$th variable. By $\nabla$ denote a gradient of a function (a vector field whose components are partial derivatives).

Consider a zero set of the field $G$
$$
G^{-1}(0)=\{x\in F\,|\,G(x)=0\}\;.
$$
With probability one $G^{-1}(0)$ is a compact  smooth $(d-1)$-dimensional submanifold in $\Rd$, i.e., a compact smooth surface.

The problem we are interested in is a calculation of average area of the surface $G^{-1}(0)$. Substituting $G/\sigma$ for $G$ does not change $G^{-1}(0)$. Therefore we may assume that $\sigma\equiv1$. We prove that
\begin{equation}\label{1805}
\E\lambda_{d-1}[G^{-1}(0)]=\frac{1}{\sqrt{2\pi}}\int_F\,e^{-m^2(x)/2}\E\big\|\nabla G(x)\big\|\,dx.
\end{equation}
For this purpose we derive an auxiliary formula for area of a surface generated by zeros of a non-random smooth field $g(x):F\to\R$:
\begin{equation}\label{1806}
\lambda_{d-1}[g^{-1}(0)]=\frac{1}{2\pi}\int_{-\infty}^{\infty}\,du\int_F\,\cos[ug(x)]\,\big\|\nabla g(x)\big\|\,dx.
\end{equation}
Before we proceed with the exact results formulation, we need to define the notion of area. There exist several well-known definitions of area of a $(d-1)$-dimensional submanifold in $\Rd$: a surface Lebesgue measure, a Hausdorff measure, a Favard measure. In general they are not equivalent. However in case of compact $\mathcal{C}^1$-smooth manifolds all three definitions coincide. Therefore we may choose any one. To prove \eqref{1806} the best choice for $\lambda_{d-1}$ is a Favard measure (for exact definition see Sect.~\ref{1037}). If $d=1$, then by $\lambda_0(A)$ we denote the cardinality of a set $A$ (may be infinite).

Recall that $F$ is supposed to be compact and $\lambda_{d-1}[\partial F]<\infty$.
\begin{theorem}\label{t1}
Suppose  $g\in\mathcal{C}^1(F)$ and
\begin{itemize}
\item[(a)] $\lambda_{d-1}[(\nabla g)^{-1}(0)]<\infty$;
\item[(b)] $\lambda_{d-1}[g^{-1}(0)\cap\partial F]=0$.
\end{itemize}
Then \eqref{1806} holds.
\end{theorem}
\begin{remark}
The proof of Theorem~\ref{t1} shows that it is possible to get rid of condition (b). Then \eqref{1806} becomes
$$
\lambda_{d-1}[g^{-1}(0)]-\frac12\lambda_{d-1}[g^{-1}(0)\cap\partial F]=\frac{1}{2\pi}\int_{-\infty}^{\infty}\,du\int_F\,\cos[ug(x)]\,\big\|\nabla g(x)\big\|\,dx\;.
$$
We shall not exploit this generalization at a later stage. 
\end{remark}

\begin{theorem}\label{t2}
Suppose a random field $G\in\mathcal{C}^1(F)$ a.s. and
\begin{itemize}
\item[(a')] $\E\lambda_{d-1}\Big[\Big(\nabla \frac{G}{\sigma}\Big)^{-1}(0)\Big]<\infty$;
\item[(b')] $\sigma(x)>0$ for all $x\in F$.
\end{itemize}
Then
\begin{equation}\label{0909}
\E\lambda_{d-1}[G^{-1}(0)]=\frac{1}{\sqrt{2\pi}}\int_F\,\exp\left\{-\frac{m^2(x)}{2\sigma^2(x)}\right\}\E\bigg\|\nabla \frac{G(x)}{\sigma(x)}\bigg\|\,dx\;.
\end{equation}
\end{theorem}

The proofs of the theorems are in Sect.~\ref{1325}. The auxiliary lemmas are in Sect.~\ref{1037}. The applications of Theorem~\ref{t2} are in Sect.~\ref{1345}.

\section{Applications of Theorem~\ref{t2}}\label{1345}
\subsection{Coarea formula} 
\begin{example}
Suppose a function $g$ satisfies the conditions of Theorem~\ref{t1}. Then
\begin{equation}\label{0123}
\int_{-\infty}^\infty\,\lambda_{d-1}[g^{-1}(u)]\,du=\int_F\,\|\nabla g(x)\|\,dx.
\end{equation}
\end{example}
\begin{proof}
Consider $G(x)=g(x)-\xi$, where $\xi$ is a Gaussian r.v. with $\E\xi=0$ and $\D\xi=\sigma^2$. Then \eqref{0909} becomes
$$
\frac{1}{\sqrt{2\pi\sigma^2}}\int_{-\infty}^\infty\,\lambda_{d-1}[g^{-1}(u)]e^{-\frac{u^2}{2\sigma^2}}\,du=\frac{1}{\sqrt{2\pi}}\int_F\,e^{-g^2(x)/(2\sigma^2)}\frac{\|\nabla g(x)\|}{\sigma}\,dx.
$$
To obtain  \eqref{0123} it remains to multiply both sides by $\sqrt{2\pi\sigma^2}$ and apply the Monotone convergence theorem (as $\sigma\to\infty$).
\end{proof}
Relation \eqref{0123} is called ``the coarea formula''. It was obtained by H.~Federer in \cite{hF59}.

\subsection{Centered Gaussian field}
By $\Sd$ denote a $(d-1)$-dimensional unit sphere with a Lebesgue measure $\mu_{d-1}(ds)$.
\begin{example}\label{e1}
If $G(x)$ satisfies the conditions of Theorem~\ref{t2} and $m(x)\equiv0$, then
\begin{equation}\label{1017}
\E\lambda_{d-1}[G^{-1}(0)]=\frac{\Gamma(\frac{d+1}{2})}{2\pi^{(d+1)/2}}\int_F\,dx\int_{\Sd}\,\sqrt{s^\top\Sigma(x)s}\,\mu_{d-1}(ds)\;,
\end{equation}
where $\Sigma(x)$ is a covariation matrix of $\nabla\{G(x)/\sigma(x)\}$.
\end{example}
\begin{proof}
The proof is by Lemma~\ref{l7} (see Sect.~\ref{1037}) which we apply to \eqref{0909}.
\end{proof}
\begin{remark}
Relation \eqref{1017} is easily extended to the case of $m(x)\equiv u, \sigma(x)\equiv1$:
\begin{equation}\label{0833}
\E\lambda_{d-1}[G^{-1}(0)]=\frac{\Gamma(\frac{d+1}{2})}{2\pi^{(d+1)/2}}\int_F\,e^{-u^2/2}\,dx\int_{\Sd}\,\sqrt{s^\top\Sigma(x)s}\,\mu_{d-1}(ds).
\end{equation}
\end{remark}

\begin{corollary}
Under the conditions of Example~\ref{e1}
\begin{multline}\label{1838}
\E\lambda_{d-1}[G^{-1}(0)]\\
=\frac{\Gamma(\frac{d+1}{2})}{2\pi^{(d+1)/2}}\int_F\,dx\int_{\Sd}\,\sigma^{-1}\bigg[\sum_{i,j=1}^d(\E G_i'G_j'-\sigma_i'\sigma_j')s_is_j\bigg]^{1/2}\,\mu_{d-1}(ds)\;.
\end{multline}
\end{corollary}
\begin{proof}
The proof follows from the fact that
$$
\Sigma=\bigg(\frac{\E G_i'G_j'-\sigma_i'\sigma_j'}{\sigma^2}\bigg)_{i,j=1}^d\;.
$$
\end{proof}

\subsection{Linear Gaussian field}\label{ss0958}
\begin{example}\label{0928}
Suppose $G(x)=\langle h(x),\xi\rangle$, where $h=(h^1,\dots,h^n)^\top:F\to\mathbb{R}^n$ is a vector function from the class $\mathcal{C}^1(F)$ and $\xi$ is a $n$-dimensional centered Gaussian vector with the identity covariation matrix. Then
\begin{equation}\label{1214}
\E\lambda_{d-1}[G^{-1}(0)]=\frac{\Gamma(\frac{d+1}{2})}{2\pi^{(d+1)/2}}\int_F\,dx\int_{\Sd}\,\|J_h(x)s\|\,\mu_{d-1}(ds)\;,
\end{equation}
where $J_h$ is the Jacobian $n$-by-$d$ matrix of  $h/\|h\|$.
\end{example}
\begin{proof}
We have $\Sigma=J_h^\top J_h$ in \eqref{1017}.
\end{proof}
\begin{remark}
If we consider a centered Gaussian vector with an arbitrary covariation matrix $\Lambda$, then $\Sigma=J_h^\top\Lambda J_h$ and \eqref{1017} becomes
$$
\E\lambda_{d-1}[G^{-1}(0)]=\frac{\Gamma(\frac{d+1}{2})}{2\pi^{(d+1)/2}}\int_F\,dx\int_{\Sd}\,\sqrt{s^\top J_h^\top(x)\Lambda J_h(x)s}\,\mu_{d-1}(ds)\;.
$$
For $d=1$ this formula was obtained by  A.~Edelman and E.~Kostlan in \cite[Theorem~3.1]{EK95}.
\end{remark}
\begin{corollary}
Suppose under the conditions of Example~\ref{0928} the rank of $J_h$ equals $k$. By $\seq{\sigma}{k}$ denote the nonzero singular values of the matrix $J_h$, i.e., the nonnegative square roots of the eigenvalues of $J_hJ_h^\top$. Then
$$
\E\lambda_{d-1}[G^{-1}(0)]=\frac{\Gamma(\frac{d+1}{2})}{2\pi^{(d+1)/2}}\int_F\,dx\int_{\Sd}\,\bigg(\sum_{j=1}^k\sigma_j(x)s_j^2\bigg)^{1/2}\,\mu_{d-1}(ds)\;.
$$
\end{corollary}
\begin{proof}
It is known from linear algebra (see, e.g., \cite{HD89}) that the matrix $J_h$ may be written in the singular form $J_h=VQ W$, where $V,W$ are $n$-by-$n$ and $d$-by-$d$ unitary matrices. The $n$-by-$d$ matrix $Q$ is diagonal. The diagonal elements are the singular values of the matrix $J_h$. We have
$$
\|J_hs\|=\|VQ Ws\|=\|Q Ws\|\;.
$$
To conclude the proof, it remains to apply this to \eqref{1214} and make a change of variables $s'=Ws$.
\end{proof}
Now we derive another form of $\E\lambda_{d-1}[G^{-1}(0)]$ which will be useful for us later.
\begin{example}
Under the conditions of Example~\ref{0928}
\begin{multline}\label{1940}
\E\lambda_{d-1}[G^{-1}(0)]=\frac{\Gamma(\frac{d+1}{2})}{2\pi^{(d+1)/2}}\\
\times\int_F\,dx\int_{\Sd}\,\bigg(\sum_{i,j=1}^d\frac{\|h\|^2\langle h_i',h_j'\rangle-\langle h,h_i'\rangle\langle h,h_j'\rangle}{\|h\|^4}s_is_j\bigg)^{1/2}\mu_{d-1}\,(ds)\;,
\end{multline}
where
$$
h_i'=\bigg(\frac{\partial h^1}{\partial x_i},\dots,\frac{\partial h^n}{\partial x_i}\bigg)^\top\;.
$$
\end{example}
\begin{proof}
We have
$$
\sigma=\|h\|\;,\quad\E G_i'G_j'=\langle h_i',h_j'\rangle\;,\quad \sigma_i'=\|h\|^{-1}\langle h,h_i'\rangle\;.
$$
It remains to apply \eqref{1838}.
\end{proof}

\subsection{Zeros of random polynomial}\label{ss0959}
\begin{example}
Consider $G(t)=\xi_{n}t^n+\dots+\xi_1t+\xi_0, t\in F\subset\mathbb{R}$, where $\{\xi_i\}$ are independent standard Gaussian random variables. Then
$$
\E\lambda_{0}[G^{-1}(0)]=\frac{1}{\pi}\int_F\,\frac{[A_n(t)C_n(t)-B_n^2(t)]^{1/2}}{A_n(t)}\,dt\;,
$$
where
$$
A_n(t)=\sum_{j=0}^nt^{2j}\;,\quad B_n(t)=\sum_{j=0}^njt^{2j-1}\;,\quad C_n(t)=\sum_{j=0}^nj^2t^{2j-2}\;.
$$
\end{example}
\begin{proof}
The proof follows from \eqref{1940}.
\end{proof}
This formula was obtained by M.~Kac in \cite{mK43}. He also derived the asymptotic relation
$$
\E\lambda_{0}[G^{-1}(0)]=\frac2\pi\log n\cdot(1+o(1)),\quad n\to\infty\,,
$$
for $F=[-\infty,\infty]$.

\subsection{Random algebraic surface} 
\begin{example}
Consider $G(x)=\sum_{\alpha}\xi_{\alpha}x^{\alpha}$, where $\alpha  = (\seq{\alpha}{d})$ is a multi-index, the summation is taken over all $\alpha$such that $0\leqslant\alpha_j\leqslant n$, and $\xi_\alpha$ are independent standard Gaussian random variables. Then
\begin{multline}\label{0957}
\E\lambda_{d-1}[G^{-1}(0)]\\
=\frac{\Gamma(\frac{d+1}{2})}{2\pi^{(d+1)/2}}\int_F\,dx\int_{\Sd}\,\bigg(\sum_{i=1}^d\frac{A_n(x_i)C_n(x_i)-B_n^2(x_i)}{A_n^2(x_i)}s_i^2\bigg)^{1/2}\,\mu_{d-1}(ds)\;.
\end{multline}
\end{example}
\begin{proof}
Using the notations of Subsection~\ref{ss0958}, we get
$$
\|h(x)\|^2=\sum_{\alpha}x^{2\alpha}=\prod_{k=1}^dA_n(x_k)\;,
$$
$$
\langle h(x),h'_i(x)\rangle=\frac12\frac{\partial}{\partial x_i}\|h(x)\|^2=B_n(x_i)\prod_{k\ne i}A_n(x_k)
$$
and
$$
\langle h_i'(x),h'_j(x)\rangle=\sum_{\alpha}\alpha_ix^{\alpha-\epsilon_i}\alpha_jx^{\alpha-\epsilon_j}=
\begin{cases}
      B_n(x_i)B_n(x_j)\prod_{k\ne i,j}A_n(x_k)&\text{for $i\ne j$,}\cr
      C_n(x_i)\prod_{k\ne i}A_n(x_k)&\text{for $i=j$,}
\end{cases}
$$
where $\epsilon_i$ denotes the multi-index in which the $i$-th position is occupied by one and all the other positions are occupied by zeros.  These relations imply that for $i\ne j$
$$
\|h\|^2\langle h_i',h_j'\rangle-\langle h,h_i'\rangle\langle h,h_j'\rangle=0
$$
and for $i=j$
$$
\|h\|^2\langle h_i',h_j'\rangle-\langle h,h_i'\rangle\langle h,h_j'\rangle=\|h\|^4\frac{A_n(x_i)C_n(x_i)-B_n^2(x_i)}{A_n^2(x_i)}.
$$
It remains to apply \eqref{1940}.
\end{proof}
Formula  \eqref{0957} was obtained by I.A.~Ibragimov and S.S.~Podkorytov in \cite{IP95}. They also derived the asymptotic relation
$$
\E\lambda_{d-1}[G^{-1}(0)]=\frac{\log d}\pi\lambda_{d-1}[F\cap\Gamma]\cdot(1+o(1)),\quad n\to\infty\;,
$$
where
$$
\Gamma=\bigcup_{j=1}^d\{x\,\Big|\,|x_j|=1\},
$$
provided that $\lambda_{d-1}[\partial F\cap\Gamma]=0$.

\subsection{Random surface of Kostlan-Shub-Smale}  
\begin{example}
Consider $G(x)=\sum_{\alpha}\xi_{\alpha}x^{\alpha}$, where the summation is taken over all nonnegative $\alpha$ such that $\alpha_1+\dots+\alpha_d\leqslant n$ and  $\xi_\alpha$ are independent Gaussian random variables with $\E\xi_\alpha=0$ and $\D\xi_\alpha=C_n^\alpha$, where
$$
C_n^\alpha=\frac{n!}{\alpha_1!\dots\alpha_d!(n-\alpha_1-\dots-\alpha)!}\;.
$$
Then
$$
\E\lambda_{d-1}[G^{-1}(0)]=\sqrt{n}\frac{\Gamma(\frac{d+1}{2})}{2\pi^{(d+1)/2}}\int_F\,\frac{dx}{1+\|x\|^2}\int_{\Sd}\,\sqrt{1+\|x\|^2-\langle x,s\rangle^2}\,\mu_{d-1}(ds)\;.
$$
\end{example}
\begin{proof}
Using the notations of Subsection~\ref{ss0958}, we get
$$
\|h(x)\|^2=\sum_{\alpha}C_n^\alpha x^{2\alpha}=\bigg(1+\sum_{k=1}^dx_k^2\bigg)^n\;,
$$
$$
\langle h(x),h_i'(x)\rangle=\frac12\frac{\partial}{\partial x_i}\|h(x)\|^2=nx_i\bigg(1+\sum_{k=1}^dx_k^2\bigg)^{n-1}\;.
$$
For $i\ne j$
\begin{multline*}
\langle h_i'(x),h_j'(x)\rangle=\sum_{\alpha}C_n^\alpha\alpha_ix^{\alpha-\epsilon_i}\alpha_jx^{\alpha-\epsilon_j} \\=n(n-1)x_ix_j\sum_{\alpha}C_{n-2}^{\alpha-\epsilon_i-\epsilon_j}x^{2\alpha-2\epsilon_i-2\epsilon_j}=n(n-1)x_ix_j \bigg(1+\sum_{k=1}^dx_k^2\bigg)^{n-2}
\end{multline*}
and for $i=j$
\begin{multline*}
\langle h'_i(x),h'_j(x)\rangle=\sum_{\alpha}C_n^\alpha\alpha_ix^{\alpha-\epsilon_i}\alpha_ix^{\alpha-\epsilon_i}\\ =\sum_{\alpha}C_n^\alpha\alpha_ix^{2\alpha-2\epsilon_i}+\sum_{\alpha}C_n^\alpha\alpha_i(\alpha_i-1)x^{2\alpha-2\epsilon_i}
\\=n\sum_{\alpha}C_{n-1}^{\alpha-\epsilon_i}x^{2\alpha-2\epsilon_i} +n(n-1)x_i^2\sum_{\alpha}C_{n-2}^{\alpha-2\epsilon_i}x^{2\alpha-4\epsilon_i}
\\=n\bigg(1+\sum_{k=1}^dx_k^2\bigg)^{n-1}+n(n-1)x_i^2\bigg(1+\sum_{k=1}^dx_k^2\bigg)^{n-2}\;.
\end{multline*}
These relations imply that for $i\ne j$
$$
\|h\|^2\langle h'_i,h'_j\rangle-\langle h,h_i'\rangle\langle h,h_j'\rangle=-n\bigg(1+\sum_{k=1}^dx_k^2\bigg)^{2n-2}x_ix_j
$$
and for $i=j$
$$
\|h\|^2\langle h'_i,h'_j\rangle-\langle h,h_i'\rangle\langle h,h_j'\rangle=n\bigg(1+\sum_{k=1}^dx_k^2\bigg)^{2n-2}\bigg(1+\sum_{k\ne i}^dx_k^2\bigg)\;.
$$
Therefore, using \eqref{1940} we get
\begin{multline*}
\E\lambda_{d-1}[G^{-1}(0)]=\frac{\Gamma(\frac{d+1}{2})}{2\pi^{(d+1)/2}}\sqrt{n}\int_F\,\bigg(1+\sum_{k=1}^dx_k^2\bigg)^{-1}\,dx
\\\times\int_{\Sd}\,\bigg(-\sum_{i\ne j}^dx_ix_js_is_j+\sum_{i=1}^d\bigg(1+\sum_{k\ne i}^dx_k^2\bigg)s_i^2\bigg)^{1/2}\,\mu_{d-1}(ds)
\\=\frac{\Gamma(\frac{d+1}{2})}{2\pi^{(d+1)/2}}\sqrt{n}\int_F\,(1+\|x\|^2)^{-1}\,dx\int_{\Sd}\,\sqrt{1+\|x\|^2-\langle x,s\rangle^2}\,\mu_{d-1}(ds)\;.
\end{multline*}
\end{proof}
\begin{remark}
Thus,
$$
\E\lambda_{d-1}[G^{-1}(0)]=C_F\sqrt{n}\;,
$$
where $C_F$ depends only on $F$ and $d$. M.~Shub and S.~Smale obtained a similar result for the number of zeros of a system of $d$ polynomials in \cite{SS93}.
\end{remark}
\begin{corollary}
For $d=1$ we get
$$
\E\lambda_{0}[G^{-1}(0)]=\sqrt{n}\int_F\,\frac{dx}{\pi(1+x^2)}\;.
$$
\end{corollary}
This relation was obtained by E.~Kostlan in \cite{eK93}.
\subsection{Random trigonometric surface}
By $|F|$ denote a volume of $F$ (i.e., a Lebesgue measure in $\Rd$).
\begin{example}
Consider
$$
G(x)=\sum_{\alpha}\big[\xi_\alpha\cos\langle\alpha,x\rangle+\eta_\alpha\sin\langle\alpha,x\rangle\big]\;,
$$
where the summation is taken over all $\alpha$ such that $0\leqslant\alpha_j\leqslant n$ and $\xi_\alpha,\eta_\alpha$are independent standard Gaussian random variables. Then
$$
\E\lambda_{d-1}[G^{-1}(0)]=n\frac{\Gamma(\frac{d+1}{2})}{4\pi^{(d+1)/2}}|F|\int_{\Sd}\,\bigg((s_1+\dots+s_d)^2+\frac{n+2}{3n}\bigg)^{1/2}\,\mu_{d-1}(ds)\;.
$$
\end{example}
\begin{proof}
Using the notations of Subsection~\ref{ss0958}, we get
$$
\|h(x)\|^2=(n+1)^d\;,\quad\langle h(x),h'_i(x)\rangle=\frac12\frac{\partial}{\partial x_i}\|h(x)\|^2=0
$$
and
$$
\langle h_i'(x),h'_j(x)\rangle=\sum_{\alpha}\alpha_i\alpha_j=
\begin{cases}
      (n+1)^{d-2}\Big(\frac{n(n+1)}{2}\Big)^2&\text{for $i\ne j$,}\cr
      (n+1)^{d-1}\frac{n(n+1)(2n+1)}{6}&\text{for $i=j$.}
\end{cases}
$$
It remains to apply \eqref{1940}.
\end{proof}

\begin{corollary}[1]
$$
\E\lambda_{d-1}[G^{-1}(0)]=c_d|F|n\cdot(1+o(1))\;,\quad n\to\infty\;,
$$
where $c_d$ depends only on the dimension $d$.
\end{corollary}
\begin{corollary}[2]
For $d=1$ we get
$$
\E\lambda_{0}[G^{-1}(0)]=\frac1\pi|F|\sqrt{\frac{n(2n+1)}{6}}\;.
$$
\end{corollary}
This formula was obtained by C.~Qualls in \cite{cC70}.

\subsection{Level sets of homogeneous Gaussian field}
\begin{example}\label{e1331}
Let $G(x)$ be a homogeneous Gaussian field with a spectral measure $\nu$. Suppose $\nu$ satisfies the conditions of Theorem~\ref{t1}.  For the sake of simplicity, we assume that $m(x)\equiv0$ and $\sigma(x)\equiv1$. Then
$$
\E\lambda_{d-1}[G^{-1}(u)]=\frac{\Gamma(\frac{d+1}{2})}{2\pi^{(d+1)/2}}|F|e^{-u^2/2}\int_{\Sd}\,\bigg(\int_{\Rd}\,\langle s,z\rangle^2\nu(dz)\bigg)^{1/2}\,\mu_{d-1}(ds)\;.
$$
\end{example}
\begin{proof}
By the spectral representation theorem,
$$
\E G(x)G(y)=\int_{\Rd}e^{i\langle x-y,z\rangle}\nu(dz)\;.
$$
Differentiating  this twice and putting $x=y=0$, we get
$$
\E G_i'(0)G_j'(0)=\int_{\Rd}z_iz_j\nu(dz)\;.
$$
Applying \eqref{0833} to $G(x)-u$, we obtain
\begin{multline*}
\E\lambda_{d-1}[G^{-1}(u)]\\
=\frac{\Gamma(\frac{d+1}{2})}{2\pi^{(d+1)/2}}e^{-u^2/2}|F|\int_{\Sd}\,\bigg(\sum_{i,j=1}^d\,s_is_j\int_{\Rd}\,z_iz_j\nu(dz)\bigg)^{1/2}\,\mu_{d-1}(ds)
\\=\frac{\Gamma(\frac{d+1}{2})}{2\pi^{(d+1)/2}}|F|e^{-u^2/2}\int_{\Sd}\,\bigg(\int_{\Rd}\,\langle s,z\rangle^2\nu(dz)\bigg)^{1/2}\,\mu_{d-1}(ds)\;.
\end{multline*}
\end{proof}
\begin{corollary}[1]
We have
$$
\frac1{\pi}\gamma_1e^{-u^2/2}|F|\leqslant\E\lambda_{d-1}[G^{-1}(0)]\leqslant\frac{\Gamma(\frac{d+1}{2})}{\sqrt{\pi}\Gamma(\frac{d}{2})}\gamma_2e^{-u^2/2}|F|\,,
$$
where
$$
\gamma_k=\bigg(\int_{\R}\|z\|^k\,\nu(dz)\bigg)^{1/k}\;.
$$

\end{corollary}
\begin{proof}
By Jensen's inequality, Fubini's theorem and Lemma~\ref{l6} (see Sect.~\ref{1037}), we get
\begin{multline*}
\int_{\Sd}\,\bigg(\int_{\Rd}\,\langle s,z\rangle^2\,\nu(dz)\bigg)^{1/2}\,\mu_{d-1}(ds)\geqslant\int_{\Sd}\,\mu_{d-1}(ds)\int_{\Rd}\,|\langle s,z\rangle|\,\nu(dz)
\\=\int_{\Rd}\,\nu(dz)\int_{\Sd}\,|\langle s,z\rangle|\,\mu_{d-1}(ds)=\int_{\Rd}\,\frac{2\pi^{(d-1)/2}}{\Gamma(\frac{d+1}{2})}\|z\|\,\nu(dz)=\frac{2\pi^{(d-1)/2}}{\Gamma(\frac{d+1}{2})}\gamma_1\;.
\end{multline*}
On the other hand, it follows from the Cauchy–-Schwarz inequality that $\|\langle s,z\rangle\|\leqslant\|s\|\|z\|=\|z\|$. Therefore,
$$
\int_{\Sd}\,\bigg(\int_{\Rd}\,\langle s,z\rangle^2\,\nu(dz)\bigg)^{1/2}\,\mu_{d-1}(ds)\leqslant\omega_{d-1}\bigg(\int_{\Rd}\,\|z\|^2\,\nu(dz)\bigg)^{1/2}=\frac{2\pi^{d/2}}{\Gamma(\frac d2)}\gamma_2\;.
$$
\end{proof}

\begin{corollary}[2]
For $d=1$ we get
$$
\E\lambda_{0}[G^{-1}(u)]=\frac{\gamma_2}{\pi}e^{u^2/2}|F|\;.
$$
\end{corollary}
This formula was obtained by S.~O.~Rice in \cite{sR45}.

\section{Auxiliary lemmas}\label{1037}
Let us recall that to define a $(d-1)$-dimensional Favard measure of a set $A$, project it onto a $(d-1)$-dimensional linear hyperplane, take the Lebesgue measure (counting multiplicities), average over all such projections, and normalize properly:
\begin{equation}\label{2148}
\lambda_{d-1}[A]=\frac{\Gamma(\frac{d+1}{2})}{2\pi^\frac{d-1}{2}}\int_{\Sd}\,\mu_{d-1}(ds)\int_{s^\perp}\,\lambda_0\Big[\{s_y^\perp\}^\perp\cap A\Big]\,dy\;,
\end{equation}
where $s^\perp$ is the linear hyperplane orthogonal to the unit vector $s\in\Sd$ and $\{s_y^\perp\}^\perp$ is the line through $y\in s^\perp$ orthogonal to $s^\perp$.

Let us introduce the notations which we shall use in this section. Put
$$
M=\sup_{R>0}\bigg|\int_{-R}^R\,\frac{\sin u}{u}\,du\bigg|\;.
$$
It follows from Lemma~\ref{l1} (see below) that $M<\infty$. By $\omega_k$ denote area of a $k$-dimensional sphere:
$$
\omega_k=\frac{2\pi^{(k+1)/2}}{\Gamma(\frac {k+1}2)}\;.
$$
Throughout this section we assume that a function $g$ satisfies the conditions of Theorem~\ref{t1}. By $g_s'$ denote a partial derivative of $g$ with respect to the direction $s\in\Sd$.
\begin{lemma}\label{l1}
For all $t\in\mathbb R$
\begin{equation}\label{1610}
\frac1\pi\int_{-\infty}^\infty\,\frac{\sin tu}u\,du=\sign t\;.
\end{equation}
\end{lemma}
\begin{proof}
See, i.e., \cite{aE80}.
\end{proof}

\begin{lemma}\label{l6}
For all $x\in\Rd$
\begin{equation}\label{2252}
\int_{\Sd}\,|\langle x,s\rangle|\,\mu_{d-1}(ds)=\frac{2\pi^{(d-1)/2}}{\Gamma(\frac{d+1}{2})}\|x\|\;.
\end{equation}
\end{lemma}
\begin{proof}
Omit the trivial case when $x=0$. Consider a Borel set $A$ such that $A\subset x^\perp$ and $\lambda_{d-1}[A]=\|x\|$. Let us apply \eqref{2148}. It is clear that the integrand $\int_{s^\perp}\lambda_0[\{s_y^\perp\}^\perp\cap A]dy$ is equal to area of the projection of $A$ onto the linear hyperplane $s^\perp$. On the other hand, if we project a set from one hyperplane to another, then area of the set multiplies by the cosine of the angle between the hyperplanes. Therefore,
$$
\int_{s^\perp}\,\lambda_0\Big[\{s_y^\perp\}^\perp\cap A\Big]\,dy=\lambda_{d-1}[A]\cdot\bigg|\bigg\langle \frac{x}{\|x\|},s\bigg\rangle\bigg|=|\langle x,s\rangle|\;.
$$
Applying this to \eqref{2148} and replacing  $\lambda_{d-1}[A]$ by $\|x\|$, we obtain \eqref{2252}.
\end{proof}

The next lemma is due to M.~Kac (see, e.g., \cite{mK03r}).
\begin{lemma}\label{l2}
If $f(t)$ continuous for $a\leqslant t\leqslant b$ and continuously differentiable for $a<t<b$ has a finite number of turning points (i.e., only a finite number of points at which $f^\prime(t)$ vanishes in (a,b)) then the number of zeros of $f(t)$ in $(a,b)$ is given by the formula
\begin{equation}\label{Kac}
\lambda_0[f^{-1}(0)]=\frac{1}{2\pi}\int_{-\infty}^{\infty}\,du\int_a^b\,\cos[uf(t)]\,|f'(t)|\,dt\;.
\end{equation}
Multiple zeros are counted once and if either $a$ or $b$ is a zero it counted as $1/2$.
\end{lemma}
\begin{remark}
This statement can be easily extended to the case of the union of a finite number of intervals. We shall use this form in the sequel.
\end{remark}
\begin{proof}
For the readers convenience we present the proof from \cite{mK03r}. Let $\seq{\alpha}{k}$ be the abscissas of the turning points:
$$
a=\alpha_0\leqslant\alpha_1 < \dots <\alpha_k\leqslant\alpha_{k+1}=b\;.
$$
We have
\begin{multline*}
\int_a^b\,\cos[uf(t)]\,|f'(t)|\,dt=\sum_{j=0}^{k}\,\int_{\alpha_j}^{\alpha_{j+1}}\,\cos[uf(t)]\,|f'(t)|\,dt
\\=\sum_{j=0}^{k}\,\bigg\{\pm\int_{\alpha_j}^{\alpha_{j+1}}\,\cos[uf(t)]\,f'(t)\,dt\bigg\}=\sum_{j=0}^{k}\,\bigg\{\pm\frac{\sin[uf(\alpha_{j+1})]-\sin[uf(\alpha_{j})]}{u}\bigg\}\;,
\end{multline*}
where the sign $+$ is attached if $f(t)$ is increasing between $\alpha_j$ and $\alpha_{j+1}$ and the sign $-$ if it is decreasing. Thus using \eqref{1610} we have
\begin{multline*}
\frac{1}{2\pi}\int_{-\infty}^{\infty}\,du\int_a^b\,\cos[uf(t)]\,|f'(t)|\,dt
\\=\sum_{j=0}^{k}\,\bigg\{\pm\frac{1}{2\pi}\int_{-\infty}^{\infty}\,\frac{\sin[uf(\alpha_{j+1})]-\sin[uf(\alpha_{j})]}{u}\,du\bigg\}
\\=\sum_{j=0}^{k}\,\bigg\{\pm\frac{\sign f(\alpha_{j+1})-\sign f(\alpha_{j})}{2}\bigg\}=\lambda_0[f^{-1}(0)]\;.
\end{multline*}
\end{proof}

\begin{lemma}\label{l3}
If $f(t)$ continuous for $a\leqslant t\leqslant b$ and continuously differentiable for $a<t<b$ has $k$ turning points, then uniformly for $R>0$
$$
\bigg|\int_{-R}^{+R}\,du\int_a^b\,\cos[uf(t)]\,|f'(t)|\,dt\bigg|\leqslant2M(k+1)\;.
$$
\end{lemma}
\begin{proof}
In the same way as in Lemma~\ref{l2} we have
\begin{multline*}
\bigg| \int_{-R}^{+R}\,du\int_a^b\,\cos[uf(t)]\,|f'(t)|\,dt\bigg|
\\=\bigg|\sum_{j=0}^{k}\,\bigg\{\pm\int_{-R}^R\,\frac{\sin[uf(\alpha_{j+1})]-\sin[uf(\alpha_{j})]}{u}\,du\bigg\}\bigg|
\\=\bigg|\sum_{j=0}^{k}\,\pm\bigg\{\int_{-Rf(\alpha_{j+1})}^{+Rf(\alpha_{j+1})}\,\frac{\sin u}{u}du-\int_{-Rf(\alpha_{j})}^{+Rf(\alpha_{j})} \frac{\sin u}{u}\,du\bigg\}\bigg|
\\\leqslant 2(k+1)\sup_{t\in\mathbb{R}}\bigg|\int_{-t}^{+t}\,\frac{\sin u}{u}\,du\bigg|=2M(k+1)\;.
\end{multline*}
\end{proof}
\begin{corollary}
If we replace $[a,b]$ by a set $H$ consisting of the union of $l$ intervals,  then uniformly for $R>0$
\begin{equation}\label{1447}
\bigg| \int_{-R}^{+R}\,du\int_H\,\cos[uf(t)]\,|f'(t)|\,dt\bigg|\leqslant2M(k+l)\;.
\end{equation}
\end{corollary}

\begin{lemma}\label{l4}
The following inequality holds:
$$
\int_{\Sd}\,\lambda_{d-1}[{g_s'}^{-1}(0)]\,\mu_{d-1}(ds)\leqslant\omega_{d-1}\lambda_{d-1}[(\nabla g)^{-1}(0)]+\omega_{d-2}|F|.
$$
\end{lemma}
\begin{proof}
We have
\begin{multline*}
\int_{\Sd}\,\lambda_{d-1}[{g_s'}^{-1}(0)]\,\mu_{d-1}(ds)=\int_{\Sd}\,\mu_{d-1}(ds)\int_F\,\ind\{g_s'(y)=0\}\,\lambda_{d-1}(dy)
\\\leqslant\omega_{d-1}\lambda_{d-1}[(\nabla g)^{-1}(0)]+\int_{\Sd}\,\mu_{d-1}(ds)\int_{F\setminus(\nabla g)^{-1}(0)}\,\ind\{g_s'(y)=0\}\,\lambda_{d-1}(dy)\;.
\end{multline*}
It remains to estimate the second summands. If $\nabla g(y)\ne0$, then the set $S(y)=\{s\in\Sd\,|\,g_s'(y)=0\}$ is contained in a unit hypersphere of the sphere $\Sd$ orthogonal to $\nabla g(y)$. Consequently $\lambda_{d-2}[S(y)]\leqslant\omega_{d-2}$ and by Fubini's theorem,
\begin{multline*}
\int_{\Sd}\,\mu_{d-1}(ds)\int_{F\setminus(\nabla g)^{-1}(0)}\,\ind\{g_s'(y)=0\}\,\lambda_{d-1}(dy)
\\=\int_{F\setminus(\nabla g)^{-1}(0)}\,dx\int_{\Sd}\ind\{f_{s}'(x)=0\}\,\mu_{d-2}(ds)
\\=\int_{F\setminus(\nabla g)^{-1}(0)}\,\lambda_{d-2}[S(y)]\,dx\leqslant\int_{F\setminus(\nabla g)^{-1}(0)}\,\omega_{d-2}\,dx=\omega_{d-2}|F|\;.
\end{multline*}
\end{proof}

\begin{lemma}\label{l5}
For all $R>0$ 
\begin{multline}\label{1849}
\int_{\Sd}\,\mu_{d-1}(ds)\int_{\{\langle y,s\rangle=0\}}\,dy\bigg|\int_{-R}^{R}\,du\int_{\{y+ts\in F\}}\,\cos[ug(y+ts)]\,|g_t'(y+ts)|dt\bigg|
\\\leqslant2M\bigg(\omega_{d-1}\lambda_{d-1}[(\nabla g)^{-1}(0)]+\omega_{d-2}|F|+\frac{\pi^{(d-1)/2}}{\Gamma(\frac{d+1}{2})}\lambda_{d-1}[\partial F]\bigg)
\end{multline}
and
\begin{multline}\label{1850}
\bigg|\int_{-R}^{R}\,du\int_F\,\cos[ug(x)]\,\|\nabla g(x)\|\,dx\bigg|
\\\leqslant \frac{\Gamma(\frac{d+1}{2})}{\pi^\frac{d-1}{2}}M\bigg(\omega_{d-1}\lambda_{d-1}[(\nabla f)^{-1}(0)]+\omega_{d-2}|F|+\frac{\pi^{(d-1)/2}}{\Gamma(\frac{d+1}{2})}\lambda_{d-1}[\partial F]\bigg)\;.
\end{multline}
\end{lemma}
\begin{proof}
By $k(s,y)$ denote the number of zeros of $g_t'(y+ts)$ (may be infinite) in the set $\{t\,|\,y+ts\in F\}$ and by $l(s,y)$ denote the number of intervals of this set (if the set is not the union of a finite number of intervals, then we put $l(s,y)=\infty$). It follows from \eqref{1447} that
\begin{equation}\label{1452}
\bigg|\int_{R}^{R}\,du\int_{\{t\,|\,y+ts\in F\}}\,\cos[ug(y+ts)]\,|g_t'(y+ts)|\,dt\bigg|\leqslant2M\Big(k(s,y)+l(s,y)\Big)\;.
\end{equation}
If we project the set ${g_s'}^{-1}(0)$ onto the hyperplane $\{y\,|\,\langle y,s\rangle=0\}$, then $k(s,y)$ is equal to the multiplicity of the projection at the point $y$. A measure does not increase under the action of projection, therefore
$$
\int_{\{\langle y,s\rangle=0\}}\,k(s,y)\,dy\leqslant\lambda_{d-1}[{g_s'}^{-1}(0)]\;,
$$
which together with Lemma~\ref{l4} implies
\begin{equation}\label{1453}
\int_{\Sd}\,\mu_{d-1}(ds)\int_{\{\langle y,s\rangle=0\}}\,k(s,y)\,dy\leqslant\omega_{d-1}\lambda_{d-1}[(\nabla g)^{-1}(0)]+\omega_{d-2}|F|\;.
\end{equation}
Further, applying the definition of a Favard measure to the boundary of $F$, we get
\begin{equation}\label{1454}
\int_{\Sd}\,\mu_{d-1}(ds)\int_{\{\langle y,s\rangle=0\}}\,2l(s,y)\,dy=\frac{2\pi^{(d-1)/2}}{\Gamma(\frac{d+1}{2})}\lambda_{d-1}[\partial F]\;.
\end{equation}
Combining \eqref{1452}, \eqref{1453} and \eqref{1454} we obtain \eqref{1849}.

Let us prove \eqref{1850}. It follows from \eqref{2252} that
$$
\|\nabla g(x)\|=\frac{\Gamma(\frac{d+1}{2})}{2\pi^{(d-1)/2}}\int_{\Sd}\,|\langle\nabla g(x),s\rangle|\,\mu_{d-1}(ds).
$$
Consequently, using Fubini's Theorem we get
\begin{multline*}
\bigg|\int_{-R}^{R}\,du\int_F\,\cos[ug(x)]\,\|\nabla g(x)\|\,dx\bigg|=\frac{\Gamma(\frac{d+1}{2})}{2\pi^{(d-1)/2}}\\
\times\bigg|\int_{-R}^{R}\,du\int_F\,dx\int_{\Sd}\,\cos[ug(x)]\,|\langle\nabla g(x),s\rangle|\,\mu_{d-1}(ds)\bigg|=\frac{\Gamma(\frac{d+1}{2})}{2\pi^{(d-1)/2}}\\
\times\bigg|\int_{Sd}\,\mu_{d-1}(ds)\int_{\{\langle y,s\rangle=0\}}\,dy\int_{-R}^{R}\,du\int_{\{x+ts\in F\}}\,\cos[ug(y+ts)]\,|g_t'(y+ts)|\,dt\bigg|.
\end{multline*}
To complete the proof it remains to apply \eqref{1849}.
\end{proof}
\begin{lemma}\label{l7}
Consider an $n$-dimensional centered Gaussian vector $\xi$ with a covariation matrix $\Sigma$. Then
$$
\E\|\xi\|=\frac{\Gamma(\frac{d+1}{2})}{\sqrt2\pi^{d/2}}\int_{\Sd}\,\sqrt{s\Sigma s^\top}\,\mu_{d-1}(ds)\;.
$$
\end{lemma}
\begin{proof}
It follows from \eqref{2252} and Fubini's theorem that
$$
\E\|\xi\|=\frac{\Gamma(\frac{d+1}{2})}{2\pi^{(d-1)/2}}\int_{\Sd}\,\E|\langle \xi,s\rangle|\,\mu_{d-1}(ds)\;.
$$
Moreover, 
$$
\E|\langle \xi,s\rangle|=\E|\mathcal{N}(0,1)|\sqrt{\D\langle \xi,s\rangle}=\bigg(\frac{2}{\pi}\bigg)^{1/2}\sqrt{s\Sigma s^\top}\;,
$$
which completes the proof.
\end{proof}

\section{Proofs of theorems}\label{1325}
\begin{proof}[Proof of Theorem~\ref{t1}]

Using \eqref{2148} and Lemma~\ref{l2}, we get
\begin{multline*}
\lambda_{d-1}[g^{-1}(0)]=\frac{\Gamma(\frac{d+1}{2})}{4\pi^{(d+1)/2}}\int_{\Sd}\,\mu_{d-1}(ds)\int_{\{\langle y,s\rangle=0\}}\,dy
\\\times\int_{-\infty}^{\infty}\,du\int_{\{y+ts\in F\}}\,\cos[ug(y+ts)]\,|g_t'(y+ts)|\,dt
\\=\frac{\Gamma(\frac{d+1}{2})}{4\pi^{(d+1)/2}}\int_{\Sd}\,\mu_{d-1}(ds)\int_{\{\langle y,s\rangle=0\}}\,dy
\\\times\lim_{R\to\infty}\int_{-R}^{R}\,du\int_{\{y+ts\in F\}}\,\cos[ug(y+ts)]\,|g_t'(y+ts)|\,dt\;.
\end{multline*}

It follows from the choice of $F$, condition (b), and  \eqref{1849} that we may apply Lebesgue's theorem:
\begin{multline*}
\lambda_{d-1}(g^{-1}(0))=\frac{\Gamma(\frac{d+1}{2})}{4\pi^{(d+1)/2}}\lim_{R\to\infty}\int_{\Sd}\,\mu_{d-1}(ds)\int_{\{\langle x,s\rangle=0\}}\,dy
\\\times\int_{-R}^{R}\,du\int_{\{x+ts\in F\}}\,\cos[ug(x+ts)]\,|g_t'(x+ts)|\,dt.
\end{multline*}
All the domains of integration are of finite measure and the integrands are bounded. Therefore we may apply Fubini's Theorem:
\begin{multline*}
\lambda_{d-1}(g^{-1}(0))=\frac{\Gamma(\frac{d+1}{2})}{4\pi^{(d+1)/2}}\lim_{R\to\infty}\int_{-R}^{R}\,du\int_{\Sd}\,\mu_{d-1}(ds)\int_{\{\langle x,s\rangle=0\}}\,dy
\\\times\int_{\{x+ts\in F\}}\,\cos[ug(x+ts)]\,|g_t'(x+ts)|\,dt
\\=\frac{\Gamma(\frac{d+1}{2})}{4\pi^{(d+1)/2}}\lim_{R\to\infty}\int_{-R}^{R}du\int_{\Sd}\,\mu_{d-1}(ds)\int_F\,\cos[ug(x)]\,|\langle\nabla g(x),s\rangle|\,dx
\\=\frac{\Gamma(\frac{d+1}{2})}{4\pi^{(d+1)/2}}\int_{-\infty}^{\infty}\,du\int_{\Sd}\,\mu_{d-1}(ds)\int_F\,\cos[ug(x)]|\langle\nabla g(x),s\rangle|\,dx\;.
\end{multline*}
To complete the proof it remains to apply Lemma~\ref{l6}.

\end{proof}

Let us proceed to the proof of the second theorem.
\begin{proof}[Proof of Theorem~\ref{t2}]
To apply Theorem~\ref{t1} we have to show that $G$ satisfies conditions (a), (b) almost surely. It easily follows from (a') that (a) holds almost surely. Further, using (b'), Fubini's theorem, and $\lambda_{d-1}[\partial F]<\infty$, we obtain
$$
\E\lambda_{d-1}[G^{-1}(0)\cap\partial F]=\E\int_{\partial F}\,\ind\{G(y)=0\}\,d\lambda_{d-1}(y)=\int_{\partial F}\,\P\{G(y)=0\}\,d\lambda_{d-1}(y)=0\;,
$$
which implies that (b) holds a.s. 

Firts let us prove the theorem for the case when $\sigma\equiv1$. From \eqref{1806} we get
\begin{multline*}
\E\lambda_{d-1}(G^{-1}(0))=\E\frac{1}{2\pi}\int_{-\infty}^{\infty}\,du\int_F\,\cos[uG(x)]\,\|\nabla G(x)\|\,dx
\\=\frac{1}{2\pi}\E\lim_{R\to\infty}\int_{-R}^{R}\,du\int_F\,\cos[uG(x)]\,\|\nabla G(x)\,\|dx\;.
\end{multline*}
It follows from the choice of $F$, condition (a'), and  \eqref{1850} that we may apply Lebesgue's theorem:
\begin{multline*}
\E\lambda_{d-1}(G^{-1}(0))=\frac{1}{2\pi}\lim_{R\to\infty}\E\int_{-R}^{R}\,du\int_F\,\cos[uG(x)]\,\|\nabla G(x)\|\,dx
\\=\frac{1}{2\pi}\lim_{R\to\infty}\int_F\,dx\int_{-R}^{R}\,\E\Big\{\cos[uG(x)]\,\|\nabla G(x)\|\Big\}\,du\;.
\end{multline*}
We may use Fubini's Theorem in the last equality on account of
$$
|\cos[uG(x)]|\;\|\nabla G(x)\|\leqslant\|\nabla G(x)\|\leqslant\sum_{j=1}^d\;|G_j'(x)|
$$
and
$$
\E\int_{-R}^{R}\,du\int_F\,\sum_{j=1}^d\,|G_j'(x)|\,dx\leqslant2R|F|\sum_{j=1}^d\,\E\sup_{x\in F}|G_j'(x)|<\infty\;.
$$
The right-hand side is finite because the supremum of a continues Gaussian field defined on a compact is summable  (see~\cite{LS70}).

Differentiating $\sigma^2\equiv 1$, we get
$$
\frac{\partial(\E G^2)}{\partial x_i}-2\E G\frac{\partial(\E G)}{\partial x_i}=0\;.
$$
Therefore, by Kolmogorov's Theorem on differentiation of mathematical expectations with respect to a parameter (see \cite{aK98}), we have
$$
\E GG_i'=\frac12\E\frac{\partial G^2}{\partial x_i}=\frac12\frac{\partial(\E G^2)}{\partial x_i}=\E G\frac{\partial(\E G)}{\partial x_i}=\E G\E G_i'\;.
$$
In other words, $G$ does not correlate with the components of the vector $\nabla G$ which is equivalent to the independence in the Gaussian case. Thus,
\begin{multline*}
\E\Big\{\cos[uG(x)]\,\|\nabla G(x)\|\Big\}=\E\cos[uG(x)]\,\E\|\nabla G(x)\|=\Re\big\{\E e^{iuG(x)}\big\}\E\|\nabla G(x)\|
\\=\Re\big\{e^{ium(x)-u^2/2}\big\}\E\|\nabla G(x)\|=\cos[um(x)]\,e^{-u^2/2}\E\|\nabla G(x)\|\;,
\end{multline*}
which implies
$$
\E\lambda_{d-1}(G^{-1}(0))\\=\frac{1}{2\pi}\lim_{R\to\infty}\int_F\,\E\|\nabla G(x)\|\,dx\int_{-R}^{R}\,\cos[um(x)]\,e^{-u^2/2}\,du\;.
$$
Using Lebesgue's Theorem and the formula
$$
\int_{-\infty}^{\infty}\,\cos[um(x)]\,e^{-u^2/2}\,du=\sqrt{2\pi}\Re\big\{\E e^{im(x)\mathcal{N}(0,1)}\big\}=\sqrt{2\pi}e^{-m^2(x)/2}\;,
$$
we obtain
\begin{multline}\label{1619}
\E\lambda_{d-1}(G^{-1}(0))=\frac{1}{2\pi}\int_F\,\E\|\nabla G(x)\|\,dx\lim_{R\to\infty}\int_{-R}^{R}\,\cos[um(x)]\,e^{-u^2/2}\,du
\\=\frac{1}{\sqrt{2\pi}}\int_F\,e^{-m^2(x)/2}\E\|\nabla G(x)\|\,dx\;.
\end{multline}

We have proved the theorem for the case when $\sigma\equiv1$. To treat the general one consider the field $G/\sigma$. It has unit variance and its zero set coincides with the zero set of $G$. Thus to complete the proof it remains to apply \eqref{1619} to $G/\sigma$.
\end{proof}

\section{Acknowledgements}

The authors are grateful to S.V.~Ivanov, A.I.~Nazarov, E.M.~Rudo, and D.S.~Chelkak for useful discussions.

A part of the work has been done in the University of Bielefeld. The authors thank F.~G\"otze for the possibility to participate at the work of CRC 701 ``Spectral Structures and Topological Methods in Mathematics''. They are also grateful to A.~Cole  for her valuable help.

\end{document}